\documentclass[11pt,twoside]{article}
\usepackage{stmaryrd}
\usepackage{amssymb}
\usepackage{amsmath}
\usepackage{amsthm}
\usepackage{color}

\allowdisplaybreaks

\def\rr{{\mathbb R}}

\def\nn{{\mathbb N}}

\def\cm{{\mathcal M}}

\def\cx{{\mathcal X}}

\def\fz{\infty}
\def\az{\alpha}

\def\gfz{\genfrac{}{}{0pt}{}}
\def\lz{\lambda}
\def\dz{\delta}
\def\bdz{\Delta}
\def\ez{\epsilon}

\def\bz{\beta}

\def\gz{{\gamma}}

\def\wz{\widetilde}

\def\ls{\lesssim}
\def\gs{\gtrsim}

\def\tbz{{\triangle_\lz}}
\def\dmz{{dm_\lz}}

\def\riz{{R_{\Delta_\lz}}}

\def\rrp{{{\mathbb\rr}_+}}
\def\bmoz{{{\rm BMO}(\mathbb\rrp,\, dm_\lz)}}

\def\xtz{{x^{2\lz}\,dx}}
\def\ytz{{y^{2\lz}\,dy}}

\def\mxy{{m_\lz(I(x, |x-y|))}}

\def\loz{{L^1(\rrp,\, dm_\lz)}}
\def\liz{{L^{\infty}(\rrp,\,dm_\lz)}}
\def\lpz{{L^p(\rrp,\, dm_\lz)}}

\def\linz{{L^\fz(\rrp,\, dm_\lz)}}

\def\rrpi{\rrp\backslash\widetilde{\mathcal{I}}}

\def\varoz{{\mathcal V}_\rho(R_{\Delta_\lz,\,\ast})}
\def\osciz{{\mathcal O(R_{\Delta_\lz,\,\ast})}}
\def\oscizp{{\mathcal O'(R_{\Delta_\lz,\,\ast})}}

\def\dsum{\displaystyle\sum}
\def\dint{\displaystyle\int}

\def\dfrac{\displaystyle\frac}

\def\r{\right}
\def\lf{\left}

\def\beeqn{\begin{equation}}
\def\eneqn{\end{equation}}
\def\beeqns{\begin{equation*}}
\def\eneqns{\end{equation*}}
\def\beeqa{\begin{eqnarray}}
\def\eneqa{\begin{eqnarray}}
\def\beeqas{\begin{eqnarray*}}
\def\eneqas{\begin{eqnarray*}}
\def\besp{\begin{split}}
\def\ensp{\begin{split}}

\newtheorem{thm}{Theorem}[section]
\newtheorem{lem}[thm]{Lemma}
\newtheorem{cor}[thm]{Corollary}
\newtheorem{defn}[thm]{Definition}

\numberwithin{equation}{section}

\begin{document}

\arraycolsep=1pt

\title{\Large\bf  Oscillation and  variation for Riesz transform associated with Bessel operators}
\author{Huoxiong Wu, Dongyong Yang\,\footnote{Corresponding author}\, and Jing Zhang}

\date{}
\maketitle

\begin{center}
\begin{minipage}{13.5cm}\small

{\noindent  {\bf Abstract:}\  Let $\lambda>0$  and
$\triangle_\lambda:=-\frac{d^2}{dx^2}-\frac{2\lambda}{x}
\frac d{dx}$ be the Bessel operator on $\mathbb R_+:=(0,\infty)$. We
show that the oscillation operator $\mathcal{O}(R_{\Delta_{\lambda},\ast})$ and variation operator $\mathcal{V}_{\rho}(R_{\Delta_{\lambda},\ast})$ of the Riesz transform $R_{\Delta_{\lambda}}$ associated with
 $\Delta_\lambda$  are both bounded on $L^p(\mathbb R_+, dm_{\lambda})$ for $p\in(1,\,\infty)$, from $L^1(\mathbb{R}_{+},dm_{\lambda})$ to $L^{1,\,\infty}(\mathbb{R}_{+},dm_{\lambda})$, and from $L^{\infty}(\mathbb{R}_{+},dm_{\lambda})$ to $BMO(\mathbb{R}_{+},dm_{\lambda})$, where $\rho\in (2,\infty)$ and $dm_{\lambda}(x):=x^{2\lambda}dx$.  As an application, we give the corresponding $L^p$-estimates for $\beta$-jump operators and the number of up-crossing.}

\end{minipage}
\end{center}

\bigskip
\bigskip


\medskip


\medskip

\thanks{H. Wu is supported by the NNSF of China (Grant Nos. 11371295, 11471041) and the NSF of Fujian Province of China (No. 2015J01025).}

\thanks{D. Yang is supported by the NNSF of China (Grant No. 11571289) and the State Scholarship Fund of China (No. 201406315078).}

\thanks{J. Zhang is supported by the NNSF of China (Grant No. 11561067).

\section{Introduction and statement of main results\label{s1}}

Let $(\mathcal X,  \mu)$ be a measure space and ${\mathcal T}_{\ast}:=\{T_\epsilon\}_{\epsilon>0}$
 be a family of operators bounded on $L^p(\mathcal X,  \mu)$ for $p\in(1, \fz)$
such that $\lim_{\epsilon\to0}T_{\epsilon}f$ exists
in some sense.  The variation operator $\mathcal{V}_{\rho}({\mathcal T}_{\ast})$
and oscillation operator $\mathcal{O}({\mathcal T}_{\ast})$ of $\mathcal T_\ast$ are two important tools to measure the speed of this convergence
in ergodic theory; see, for example, \cite{Bou,HMMT,jkrw,jr}.
We recall that  for any $f\in\lpz$ for $p\in(1, \fz)$ and $x\in\cx$, $\mathcal{V}_{\rho}({\mathcal T}_{\ast})(f)$ and
$\mathcal{O}({\mathcal T}_{\ast})(f)$ are, respectively, defined by setting
 \begin{equation*}\label{d-var}
 \mathcal{V}_{\rho}({\mathcal T}_{\ast})(f)(x):= \sup_{\epsilon_i\searrow0}
 \Big(\sum_{i=1}^{\infty}|T_{\epsilon_{i+1}}f(x)-T_{\epsilon_{i}}f(x)|^{\rho}\Big)^{1/\rho},
 \end{equation*}
 where the supremum is taken over all sequences $ \{\epsilon_{i}\}$ decreasing to zero, and
\begin{equation*}\label{d-oci}
\mathcal{O}({\mathcal T}_{\ast})(f)(x):=\Big(\sum_{i=1}^{\infty}\sup_{\epsilon_{i+1}\leq t_{i+1}
  <t_{i}\leq \epsilon_{i}}|T_{t_{i+1}}f(x)-T_{t_{i}}f(x)|^2\Big)^{1/2}
\end{equation*}
 with $\{\epsilon_{i}\}$ being a fixed sequence decreasing to zero.
We also consider the operator
 \begin{equation*}\label{d-o'ci}
 \mathcal{O}'({\mathcal{T}}_{\ast})(f)(x)=\Big(\sum_{i=1}^{\infty}\sup_{\ez_{i+1}<\delta_{i}\le \ez_{i}}
 |T_{\ez_{i+1}}f(x)-T_{\delta_{i}}f(x)|^2\Big)^{1/2}.
 \end{equation*}
 It is easy to check that
 \begin{equation}\label{oci-o'ci}
 \mathcal{O'}({\mathcal{T}}_{\ast})(f)(x)\leq \mathcal{O}({\mathcal{T}}_{\ast})(f)(x)\leq 2\mathcal{O'}({\mathcal{T}}_{\ast})(f)(x).
 \end{equation}
Denote by $\mathcal E$  the mixed norm
Banach space of two variables function $h$ defined on $(0,\infty)\times\mathbb{N}$ such that
\begin{equation}\label{mixed norm}
\| h \|_{\mathcal E}:=\lf(\sum_i \lf(\sup_{\delta_i}|h({\delta_i},i)|\r)^2\r)^{1/2}<\infty.
\end{equation}
Then we also have that
\begin{equation}\label{oe}
\mathcal{O'}(\mathcal{T}_{\ast})(f)(x)=||\{T_{t_{i+1}}f(x)-T_{\delta_i}f(x)\}_{{\delta_i}\in (t_{i+1},\,t_i],\,i\in \mathbb{N}}||_{\mathcal E}.
\end{equation}
%

 In their remarkable work \cite{cjrw1}, Campbell {\it et al.} established the
strong $(p,\,p)$-boundedness for $p\in(1, \fz)$ and the weak
type $(1,\,1)$-boundedness of the oscillation operator
 and the $\rho$-variation operator for the Hilbert transform, and applied to the study of $\lz$-jump operator in ergodic theory.
 This result was further extended in \cite{cjrw2}, to the higher dimensional cases including  Riesz transforms and  general singular integrals with rough homogeneous kernels in $\mathbb{R}^d$. Since then, boundedness of oscillation and variation operators of singular integrals
 operators associated with differential operators has
been studied in many recent papers.
  In particular, Gillespie and Torrea in \cite{gt} established weighted $L^p(\rr,\omega)$-boundedness of the oscillation operator  and the $\rho$-variation operator for the Hilbert transform, where $\omega\in A_p(\rr)$, the Mukenhoupt class and obtained the $L^p(\mathbb{R}^d,|x|^\alpha dx)$-boundedness of the oscillation operator
 and the $\rho$-variation operator for Riesz transform, for $p\in(1, \fz)$ and $\az\in(-1, p-1)$.
 Later, Betancor {\it et al.}  \cite{bfbmt-2} showed that the oscillation operator  and the $\rho$-variation operator
 of the Riesz transform $R_{S_\lz}$ associated with the Bessel operator $S_\lz$ for $\lz>0$ on $\rrp :=(0, \fz)$ is
 bounded on $L^p(\omega)$ and from $L^1(\omega)$ to $L^{1,\,\fz}(\omega)$, where $\omega\in A^p(\rrp)$ and $S_{\lz}:=-\frac{d^2}{dx^2}+\frac{\lz^2-\lz}{x^2}.$
For more results on  variation and oscillation of singular integral operators, we refer the readers to \cite{Bou,bct,BFHR,cjrw1,cjrw2,cmtt,cmtv,gt,jr,jsw,mt,LW,ZW} and the references therein.

Inspired by the result of Betancor {\it et al.} in \cite{bfbmt-2}, the aim of this paper is to prove the
$L^p$-boundedness and their endpoint estimates of the oscillation and variation operators for
Riesz transforms  associated with $\Delta_\lz$, the conjugation of the Bessel operator.
To this end, we recall some necessary notation.

Let $\lz$ be a positive constant.
The operator $\tbz$ is defined by setting, for all suitable functions $f$ on $\rrp$,
\begin{equation*}
\tbz f(x)=-\frac{d^2}{dx^2}f(x)-\frac{2\lz}{x}\frac{d}{dx}f(x).
\end{equation*}
An early work concerning the Bessel operator is from Muckenhoupt and Stein \cite{ms}.
They  developed a theory associated to
$\tbz$ which is parallel to the classical one associated to the Laplace operator $\triangle$.
After the paper \cite{ms}, a lot of work concerning the Bessel operators was carried out; see, for example \cite{ak,bcfr,bfbmt,bfs,bhnv, dlwy, k78,v08,yy}. Among the study of $\tbz$,
the properties and $L^p$ boundedness  of Riesz transforms  associated  to $\tbz$   defined by
\begin{equation*}
R_{\Delta_\lambda}f := \partial_x (\Delta_\lz)^{-1/2}f \ ,
\end{equation*}
$(1<p<\infty)$, have been studied extensively, see for example \cite{ak,bcfr,bfbmt, ms,v08}.
In particular, in \cite[pp.\,710-711]{bfbmt}, Betancor {\it et al.} showed that if $1\le p<\infty$ and $f\in L^{p}(\rrp,\dmz)$, then for almost every $x\in \rrp$,
\begin{equation*}
R_{\Delta_\lambda}f(x) = \lim_{\varepsilon\rightarrow 0_+}R_{\Delta_\lambda,\,\varepsilon}f(x):=\lim_{\varepsilon\rightarrow 0_+}\int_{0,|x-y|>\varepsilon}^{\infty} R_{\Delta_\lambda}(x,y)f(y)\dmz(y),
\end{equation*}
where  $\dmz(y):= y^{2\lz}\,dy$ and for any $x,y\in \rrp$ with $x\not=y$,
\begin{equation*}
R_{\Delta_\lambda}(x,y) := -\frac{2\lz}{\pi}\int_{0}^{\pi}\frac{(x-y\cos\theta)(\sin\theta)^{2\lz-1}}{(x^2+y^2-2xy\cos\theta)^{\lz+1}}d\theta.
\end{equation*}
Moreover, Betancor {\it et al.} in \cite{bdt} characterized the atomic Hardy space $H^1(\mathbb{R}_+, \dmz)$
associated to $\tbz$ in terms of the Riesz transform and the radial
maximal function associated with the Hankel convolution of a class
of suitable functions.
%

Let $\rho>2$ and $R_{\bdz_{\lz},\ast} :=\{R_{\Delta_\lambda,\,\varepsilon}\}_{\varepsilon>0}$ be a family of truncated Riesz transform operators defined by
\begin{equation*}
R_{\Delta_\lambda,\,\varepsilon}f(x) := \int_{0,|x-y|>\varepsilon}^{\infty} R_{\Delta_\lambda}(x,y)f(y)\,\dmz(y).
\end{equation*}
The $\rho$-variation operator $\varoz$ and oscillation operator $\osciz$
associated with the Riesz transform are defined by setting, for all suitable functions $f$ and $x\in\rrp$,
\begin{equation*}\label{varia}
\varoz f(x):=\sup_{\varepsilon_j\searrow0}\lf(\sum_{j=1}^\fz\lf|R_{\Delta_\lz,\,\varepsilon_j}f(x)-R_{\Delta_\lz,\,\varepsilon_{j+1}}f(x)\r|^\rho\r)^{1/\rho}
\end{equation*}
with the supremum taken over all sequences $\{\varepsilon_j\}_j$  decreasing converging to zero,
 and
\begin{equation*}\label{oscia}
\osciz f(x):=\lf(\sum_{j=1}^\fz\sup_{\varepsilon_{j+1}\le t_{j+1}<t_j\le \varepsilon_j}\lf|R_{\Delta_\lz,\,t_j}f(x)-R_{\Delta_\lz,\,t_{j+1}}f(x)\r|^2\r)^{1/2},
\end{equation*}
where $\{\varepsilon_j\}_j$ is a fixed decreasing sequence converging to zero.
%

We are now to the first main result of this paper.

\begin{thm}\label{t-bdd riesz}
Let $\rho\in(2, \fz)$ and $\lambda>0$. The operators $\osciz$ and $\varoz$ are bounded from $\lpz$ into itself, for every $1<p<\infty$, and from $\loz$ to $L^{1,\,\fz}(\rrp,\,dm_\lz)$.
\end{thm}

As applications of Theorem \ref{t-bdd riesz}, we consider the $\beta$-jump operators and the number of up-crossing associated with the operators sequence $\{R_{\Delta_\lambda,\,\varepsilon}\}_{\varepsilon>0}$, which give certain quantitative information on the convergence of the family $\{R_{\Delta_\lambda,\,\varepsilon}\}_{\varepsilon>0}$.

\begin{defn}[\cite{cjrw1}]\rm
Let $\bz>0$. The $\beta$-jump operator $\Lambda(R_{\bdz_{\lz},\ast},\,f,\,\beta)(x)$ associated with a sequence $R_{\bdz_{\lz},\ast} =\{R_{\Delta_\lambda,\,\varepsilon}\}_{\varepsilon>0}$ acting on a function $f$ at a point $x$ is defined by
\begin{eqnarray*}
\Lambda(R_{\bdz_{\lz},\ast},f,\beta)(x):&=&\sup\{n\in\mathbb{N}:\quad {\rm there\quad exist} \quad s_1<t_1\leq s_2<t_2<\cdots\leq s_n<t_n\\
&& \quad\quad {\rm such\,\, that}\,\, |R_{\Delta_\lambda,\,s_i}f(x)-R_{\Delta_\lambda,\,t_i}f(x)|>\beta \quad {\rm for} \quad i=1,2,\cdots,n\}.
\end{eqnarray*}
\end{defn}
Also, for fixed $0<\alpha<\gamma$, we consider the number of up-crossing $N(R_{\bdz_{\lz},\ast},f,\alpha,\gamma,x)$ associated with a sequence $R_{\bdz_{\lz},\ast} =\{R_{\Delta_\lambda,\,\varepsilon}\}_{\varepsilon>0}$ acting on a function $f$ at a point $x$, which is defined by
\begin{eqnarray*}
N(R_{\bdz_{\lz},\ast},f,\alpha,\gamma,x)&:=&\sup\Big\{n\in\mathbb{N}:\quad {\rm there\quad exist} \quad s_1<t_1<s_2<t_2<\cdots<s_n<t_n\\
&& \quad\quad {\rm such\,\, that}\,\,  R_{\Delta_\lambda,\,s_i}f(x)<\alpha,\,R_{\Delta_\lambda,\,t_i}f(x)>\gamma \quad {\rm for}\quad i=1,2,\cdots,n\Big\}.
\end{eqnarray*}

It is easy to check that
\begin{equation}\label{N-Lambda}
N(R_{\bdz_{\lz},\ast},f,\alpha,\gamma,x)\leq \Lambda(R_{\bdz_{\lz},\ast},f,\gamma-\alpha)(x).
\end{equation}
Also, from \cite{cjrw1}, the $\beta$-jump operators is controlled by the  $\rho$-variation operator. Precisely, we have
that for any $\bz\in(0, \fz)$,
\begin{equation}\label{Lambda-Var}
\beta(\Lambda(R_{\bdz_{\lz},\ast},f,\beta)(x))^{1/\rho}\leq \mathcal{V_\rho}(R_{\bdz_{\lz},\ast}f)(x).
\end{equation}

As an immediate corollary of Theorem \ref{t-bdd riesz}, \eqref{Lambda-Var} and \eqref{N-Lambda}, we have the following result.
\begin{cor}
Let $\rho\in(2, \fz)$, $\lambda,\,\bz\in(0, \fz)$ and $0<\az<\gz$. Then there exist positive constants
$C(p,\rho,\lambda)$ and $C(\rho,\lambda)$, such that for all $f\in\lpz$ with $p\in(1, \fz)$,
\begin{eqnarray*}
\Big\|\Lambda(R_{\bdz_{\lz},\ast},f,\beta,\cdot)^{1/\rho}\Big\|_{\lpz}
 \leq\frac{C(p,\rho,\lambda)}{\beta}\|f\|_{\lpz},
\end{eqnarray*}
\begin{eqnarray*}
\Big\|N(R_{\bdz_{\lz},\ast},f,\alpha,\gamma,\cdot)^{1/\rho}\Big\|_{\lpz}
 \leq\frac{C(p,\rho,\lambda)}{\gamma-\alpha}\|f\|_{\lpz},
\end{eqnarray*}
and for any $f\in\loz$ and $n\geq 1$,
\begin{eqnarray*}
m_{\lz}(\{x\in\rr_+:\, \Lambda(R_{\bdz_{\lz},\ast},f,\beta,x)>n\})\le \frac{C(\rho,\lambda)}{\beta n^{1/\rho}}\|f\|_{\loz},
\end{eqnarray*}
and
\begin{eqnarray*}
m_{\lz}(\{x\in\rr_+:\, N(R_{\bdz_{\lz},\ast},f,\alpha,\gamma,x)>n\})\le \frac{C(\rho,\lambda)}{(\gamma-\alpha) n^{1/\rho}}\|f\|_{\loz}.
\end{eqnarray*}
\end{cor}

For $p=\fz$, we also study the boundedness of $\osciz$ and $\varoz$ from $L^{\fz}(\rr_+,\dmz)$ to $\bmoz$ introduced in \cite{bdt}.

\begin{defn}[\cite{bdt,yy}]\label{d-bmo}\rm
A function $f\in L^1_{\rm loc}(\rrp,\,dm_\lambda)$ belongs to
the {space} $\bmoz$ if
\begin{equation*}\label{mofi}
\|f\|_{{\rm BMO}(\rrp,\,\dmz)}:=\sup_{x,\,r\in (0,\,\infty)}
\frac 1{m_\lambda(I(x,r))}\int_{I(x,\,r)}|f(y)-f_{I(x,\,r),\,\lambda}|\,y^{2\lambda}dy<\infty,
\end{equation*}
where $I(x, r):=(x-r, x+r)\cap (0, \fz)$ and
\begin{equation*}\label{f-i}
f_{I(x,\,r),\,\lambda}:=\frac 1{m_{\lambda}(I(x,r))}\int_{I(x,\,r)}f(y)\,y^{2\lambda}dy.
\end{equation*}
\end{defn}

Our result concerning the boundedness of $\osciz$ and $\varoz$ for $p=\infty$ is stated as below.

\begin{thm}\label{t-bmo riesz}
Let $\rho\in(2, \fz)$ and $\lambda>0$. If $f$ in $L^\infty(\rrp,\dmz)$ and $\osciz f(x)<\infty$ a.e. $x\in \rr_+$, then
$\osciz f\in \bmoz$ and there exists a positive constant $C$ independent $f$ such that
\begin{equation*}\label{upp bdd of riesz bmo}
\|\osciz f\|_{\bmoz}\le C\|f\|_{\linz}.
\end{equation*}
The same result holds for $\varoz$.
\end{thm}

The organization of this paper is as follows.

Section \ref{L-p-1} is devoted to the proof of Theorem \ref{t-bdd riesz}. We present the proof of the boundedness of $\osciz$ by dividing into two steps. In the first step,  motivated by \cite{bct} and \cite{bfbmt-2}, we show that for any $p\in(1+2\lz, \infty)$, $\oscizp$ is bounded on $L^p(\rrp,\dmz)$.
Notice that though $x^{2\lz}\in A_p(\rrp)$ and the kernel $R_{\Delta_\lambda}(x, y)=(xy)^\lz R_{S_\lz}(x,y)$,
where $R_{S_\lz}(x,y)$ is the kernel of $R_{S_\lz}$,
one can not obtain the $\lpz$-boundedness of $\oscizp$ directly by that of $\mathcal{O}'(R_{S_\lz,\,\ast})$ in \cite{bfbmt-2}.
We mention that in this step, by decomposing the kernel of $\riz$ into four parts, we first have that for any $f\in\lpz$ and $x\in\rrp$,
 $$\oscizp f(x)\le C\lf[\mathcal{O}'(H^{loc}_{\ast}){f}(x)+T_1f(x)+T_2f(x)+\cm_{\lambda}f(x)\r],$$
 where $H^{loc}$ is the local Hilbert transform introduced by Andersen and Muchenhoupt \cite{am82},
 $\cm_\lz$ is the Hardy-Littlewood maximal operator on $\rrp$ with measure $dm_\lz$ and $T_i$, $i=1,2$, are bounded
 operators on $\lpz$. Moreover, by decomposing the Hilbert transform $H$ on $\rr$ into three parts, we further obtain that
$$\mathcal{O}'(H^{loc}_{\ast}){f}(x)\le C\lf[\mathcal{O}'(H_{\ast})\tilde{f}(x)+\mathcal{M}f(x)+T_{1}f(x)\r],$$
where $\tilde f(x):=f(x)$ if $x\in\rrp$ and $0$ otherwise, and $\cm$ is the Hardy-Littlewood maximal operator on $\rrp$ with  Lebesgue measure.
Then by the known fact that $x^{2\lz}\in A_p(\rrp)$ if and only if $0<2\lz<p-1$ and the $L^p(\omega)$-boundedness of $\mathcal{O}'(H_{\ast})$,
established by \cite{gt}, for $p\in (1, \fz)$ and $\omega\in A_p(\rr)$, we obtain the $L^p(\rrp,\dmz)$-boundedness of $\oscizp$  with $p\in(1+2\lz, \infty)$.

In the second step, by applying the Calder\'on-Zygmund decomposition established by Coifman and Weiss \cite{cw71},
and the $L^p(\rrp,\dmz)$-boundedness of $\oscizp$  with $p\in(1+2\lz, \infty)$
obtained in the first step, we establish the weak type (1,1) estimation of $\oscizp$.
 Then by the Marcinkiewicz  interpolation theorem, we further obtain the $L^p(\rrp,\dmz)$-boundedness of $\oscizp$  with $p\in(1, 1+2\lz]$.
 Via \eqref{oci-o'ci}, we then show that $\osciz$ is bounded on $\lpz$ for all $p\in(1, \fz)$ and from $\loz$ to $L^{1,\,\fz}(\rrp,\dmz)$.

 In Section \ref{s3}, we investigate behaviors of $\osciz$ and $\varoz$ at endpoint $p=\infty$. Under the assumption that
$f\in\linz$ with $\osciz f(x)<\fz$ $a.e.\, x$, we apply Theorem \ref{t-bdd riesz} and the known upper bound of kernel $R_{\Delta_{\lambda}}$ to establish the $(\liz, {\rm BMO}(\rrp,\dmz))$-boundedness of $\osciz$.

Throughout the paper,
we denote by $C$  {positive constants} which
are independent of the main parameters, but they may vary from line to
line. If $f\le Cg$, we then write $f\ls g$ or $g\gs f$;
and if $f \ls g\ls f$, we  write $f\sim g.$
For every $p\in(1, \fz)$, $p'$ means the conjugate of $p$, i.e., $1/p'+1/p=1$.
For any $k\in \mathbb{R}_+$ and $I:= I(x, r)$ for some $x$, $r\in (0, \fz)$,
$kI:=I(x, kr)$.

\section{$L^p(\rrp,\dmz)$-boundedness and weak type (1,1) estimate}\label{L-p-1}

In this section, we provide the proof of Theorem \ref{t-bdd riesz}. To begin with, we first recall
a useful lemma on the upper and lower bounds of  kernel $R_{\bdz_\lz}(y,z)$ of $R_{\bdz_\lz}$,
which is an important tool in this paper and can be found in, for example, \cite{bfbmt,bfbmt-2,dlwy}.
\begin{lem}\label{l-RieszCZ}
The kernel $\riz(y,z)$ satisfies the following conditions:
\begin{itemize}
  \item [{\rm i)}] There exists a positive constant $C$ such that for any $y,z\in\mathbb{R}_+$ with $y\not=z$,
  \begin{equation}\label{cz kernel condition-1}
  |\riz(y,z)|\le C \frac1{m_\lz(I(y, |y-z|))} .
  \end{equation}
  \item [{\rm ii)}]There exists a positive constant $\wz C$ such that  for any $y,\,y_0,\, z\in \mathbb{R}_+$ with $|y_0-z|<|y_0-y|/2$,
  \begin{eqnarray}\label{cz kernel condition-2}
   &&|\riz(y, y_0)-\riz(y,z)|+ |\riz(y_0, y)-\riz(z,y)|\nonumber\\
   &&\quad\le \wz C \frac{|y_0-z|}{|y_0-y|}\frac1{m_\lz(I(y, |y_0-y|))}.
  \end{eqnarray}

  \item [{\rm iii)}] There exist positive constants $K_1>2$ large enough and $C_{K_1,\,\lz}, \wz C_{K_1,\,\lz}\in(1, \fz)$ such that
  for any $y,\,z\in\mathbb{R}_+$ with $z>K_1y$,
  \begin{equation*}\label{lower bound of riesz kernel}
\frac y{z^{2\lz+2}}/C_{K_1,\,\lz}\le  \riz(y, z)\le \wz C_{K_1,\,\lz}\frac y{z^{2\lz+2}}.
  \end{equation*}
  \item [{\rm iv)}] There exist $K_2\in(1/2,1)$ such that $1-K_2$ small enough and a positive constant $C_{K_2,\lz}$
  such that for any $y,\,z\in\mathbb{R}_+$ with $z/y\in(K_2, 1)$,
  \begin{equation*}\label{appro estimate of riesz kernel}
 \lf |\riz(y,z)+\frac1\pi\frac1{y^\lz z^\lz}\frac1{y-z}\r|\le C_{K_2,\,\lz}\frac1{y^{2\lz+1}}\lf(\log_+\frac{\sqrt{yz}}{|y-z|}+1\r).
  \end{equation*}
 \end{itemize}
\end{lem}

It is straightforward from the definition of  $m_\lambda$ (i.e., $dm_\lambda(x):=x^{2\lambda}dx$)  that
 there exists a positive constant $C > 1$ such that for all $x,\,r\in\mathbb{R}_+$,
\begin{equation}\label{volume property-1}
 C^{-1}m_\lz(I(x, r))\le x^{2\lz}r+r^{2\lz+1}\le C m_\lz(I(x, r)).
\end{equation}
This means that $(\mathbb{R}_+, |\cdot|, dm_\lz)$ is a space of homogeneous type
in the sense of \cite{cw71,cw77}.

To establish the weak type estimation for $\osciz$ and $\varoz$, another main tool in our proof is the Calder\'on-Zygmund decomposition established in \cite[pp.\,73-74]{cw71} in the setting of spaces of homogeneous type.
\begin{lem}\label{l-cz} Let $f\in\loz$ and
$\eta>0$, there exist functions $g$ and $b$, a family
of intervals $\{I_j\}_{j}$, and constants $C>0$ and $M\ge 1$, such
that
\begin{itemize}
\item [\rm (i)] $f=g+b=:g+\sum_jb_j$, where $b_j$ is supported in $I_j$,
\item [\rm (ii)] $\frac{1}{m_{\lz}(I_j)}\int_{I_j}|b_j|\dmz\le C\eta$ and $\int_{I_j} b_j(x)\,\dmz(x)=0$
for each $j$,
\item [\rm (iii)] $\|g\|_\linz\le C\eta$ and $\|g\|_\loz\le C\|f\|_\loz$,
\item [\rm (iv)] $\sum_j\|b_j\|_\loz\le C\|f\|_\loz$,
\item [\rm (v)] $\sum_jm_\lz(I_j)\le \frac C\eta\|f\|_\loz$,
\item [\rm (vi)] for any $x\in\rrp$, $\sum_j\chi_{I_j}(x)\le M$.
\end{itemize}
\end{lem}

\begin{proof}[Proof of Theorem \ref{t-bdd riesz}]
We only give the estimation $\osciz$. The proof for $\varoz$  can be given analogously as that of $\osciz$, and we leave the part to the interested readers. Moreover, by \eqref{oci-o'ci}, it suffices to prove that $\oscizp$ is bounded on $\lpz$ for all $p\in (1, \fz)$ and
from $\loz$ to $L^{1,\,\fz}(\rrp,\,dm_\lz)$. We divide the proof into two steps as follows.

{\bf Step 1.} We first show that $\oscizp$ is bounded on $L^p(\rrp,\,\dmz)$ for any $p\in(1+2\lz, \infty)$.
To this end, let $\{t_j\}_j$ be a fixed sequence which decreases to zero and $\delta_j\in(t_{j+1},t_j]$ for each $j$, and
 $$B_{\delta_j,\,t_{j+1}}:=\{y\in\rrp: t_{j+1}<|x-y|\le \delta_j\}.$$
 Motivated by the method in \cite{bct} (see also \cite{bfbmt-2}), we decompose the $\lf|R_{\Delta_\lz,\,t_{j+1}}f(x)-R_{\Delta_\lz,\,\dz_{j}}f(x)\r|$  according to the domain of integration as follows£»
\begin{eqnarray*}
&&\lf|R_{\Delta_\lz,\,t_{j+1}}f(x)-R_{\Delta_\lz,\,\dz_{j}}f(x)\r|\\
&&\quad=\lf|\int_{0}^{\frac{x}{2}}R_{\Delta_\lz}(x,y)f(y)\chi_{B_{\delta_j,\,t_{j+1}}}(y)y^{2\lz}dy+\int_{2x}^{\fz}R_{\Delta_\lz}(x,y)f(y)\chi_{B_{\delta_j,\,t_{j+1}}}(y)y^{2\lz}dy\r.\\
&&\quad\quad\quad+\int_{\frac{x}{2}}^{2x}-\frac{1}{\pi}\frac{1}{x^{\lz}y^{\lz}}\frac{f(y)}{x-y}\chi_{B_{\delta_j,\,t_{j+1}}}(y)y^{2\lz}dy\\
&&\quad\quad\quad+\lf.\int_{\frac{x}{2}}^{2x}\lf[R_{\Delta_\lz}(x,y)+\frac{1}{\pi}
\frac{1}{x^{\lz}y^{\lz}}\frac{1}{x-y}\r]f(y)\chi_{B_{\delta_j,\,t_{j+1}}}(y)y^{2\lz}dy\r|\\
&&\quad=:\lf|\sum_{i=1}^4{\rm I}_i(\delta_j, j)(x)\r|.
\end{eqnarray*}

Let $\mathcal E$  be the mixed norm
Banach space of two variables function $h$ defined on $(0,\infty)\times\mathbb{N}$ satisfying \eqref{mixed norm}.
Then we see that
$$\oscizp f(x)\le \sum_{i=1}^4\lf\|{\rm I}_i(\delta_j,j)(x)\r\|_{\mathcal E}.$$
By $\|\chi_{B_{\delta_j,\,t_{j+1}}}(y)\|_{\mathcal E}\le 1$, Lemma \ref{l-RieszCZ} i), \eqref{volume property-1} and Minkowski's inequality, we obtain that
$$\lf\|{\rm I}_1(\delta_j,j)(x)\r\|_{\mathcal E}\le \frac1{x^{2\lz+1}}\dint_0^x|f(y)|\ytz\lesssim \cm_\lz(f)(x),$$
where $\cm_\lz$ is the Hardy-Littlewood maximal operator defined by setting, for any function
$f\in L^1_{\rm loc}(\rrp, dm_\lz)$ and $x\in \rrp$,
$$\cm_\lz (f)(x):=\sup_{\gfz{I\subset\rr_+}{I\ni x}}\frac1{m_\lz(I)}\int_I|f(y)|\,\ytz.$$
From \cite{cw71}, we know that $\cm_\lz$ is bounded on $L^p(\rrp,\dmz)$  with $p\in(1, \infty)$ and from $L^1(\rrp,\dmz)$
to $L^{1,\infty}(\rrp,\dmz)$.

Similarly, by iii) and iv) of Lemma \ref{l-RieszCZ}, we have
$$\lf\|{\rm I}_2(\delta_j,j)(x)\r\|_{\mathcal E}\lesssim \int_{2x}^\fz|f(y)|\frac{dy}y=: T_1(f)(x),$$
and
\begin{eqnarray*}
\lf\|{\rm I}_4(\delta_j,j)(x)\r\|_{\mathcal E}\lesssim T_2(f)(x)+\frac1{x^{2\lz+1}}\int_{\frac x2}^{2x}|f(y)|\,y^{2\lz}dy\lesssim T_2(f)(x)+\cm_\lz(f)(x),
\end{eqnarray*}
where
$$T_2f(x):= \int_{\frac x2}^{2x}\frac1{x^{2\lz+1}}\log_+\frac{\sqrt{xy}}{|x-y|}|f(y)|\ytz.$$

By change of variables and  Minkowski's inequality, we see that for any $f\in\lpz$,
\begin{eqnarray}\label{T1}
\lf\|T_1(f)\r\|_\lpz&=&\lf\{\int_0^{\fz}\lf[\int_2^\fz|f(xz)|\,\frac{dz}{z}\r]^p\xtz\r\}^{1/p}\nonumber\\
&\le&\int_2^\fz\lf[\int_0^\fz|f(xz)|^p\,\xtz\r]^{1/p}\frac{dz}{z}\nonumber\\
&= &\int_2^\fz\lf[\int_0^\fz|f(y)|^p{y^{2\lz}}\,dy\r]^{1/p}\frac{dz}{z^{(2\lz+1)/p+1}}\nonumber\\
&=&\|f\|_\lpz\int_2^\fz z^{-(2\lz+1)/p-1}\,dz\nonumber\\
&=&\frac{p2^{-(2\lz+1)/p}}{2\lz+1}\|f\|_\lpz.
\end{eqnarray}

On the other hand, observe that
$$\lf[\int_{\frac x2}^{2x}\lf(\log_{+}\frac{\sqrt{xy}}{|x-y|}\r)^{p'}y^{2\lz}\, dy\r]^{\frac{p}{p'}}\sim
\lf(\int_{\frac12}^2\lf(\log_{+}\frac{\sqrt t}{|1-t|}\r)^{p'}\,dt\r)^{\frac{p}{p'}}
x^{\frac{p(2\lz+1)}{p'}}\sim x^{\frac{p(2\lz+1)}{p'}}.$$
By this and  H\"older's inequality, we have that
\begin{eqnarray}\label{T2}
&&\lf\|T_2(f)\r\|^p_\lpz\nonumber\\
&&\quad=\int_0^\fz\frac1{x^{p(2\lz+1)}}\lf[\int_{\frac x2}^{2x}\log_{+}\frac{\sqrt{xy}}{|x-y|}|f(y)|\,y^{2\lambda}dy\r]^p\xtz\nonumber\\
&&\quad\le \int_0^\fz x^{2\lz-p(2\lz+1)}\lf[\int_{\frac{x}2}^{2x}
\lf(\log_{+}\frac{\sqrt{xy}}{|x-y|}\r)^{p'}\ytz\r]^{\frac{p}{p'}}\lf[\int_{\frac{x}2}^{2x}|f(y)|^p\,\ytz\r]\,dx\nonumber\\
&&\quad\sim\int_0^\fz x^{2\lz-p(2\lz+1)+\frac {p(2\lz+1)}{p'}}\lf[\int_{\frac{x}2}^{2x}\lf(\log_{+}\frac{\sqrt{xy}}{|x-y|}\r)^{p'}\,dy\r]^{\frac{p}{p'}}
\lf[\int_{\frac{x}2}^{2x}|f(y)|^p\,\ytz\r] dx\nonumber\\
&&\quad\sim \int_0^\fz x^{-1}\int_{\frac{x}2}^{2x}|f(y)|^p\,\ytz\,dx\nonumber\\
&&\quad\sim\int_0^\fz|f(y)|^py^{2\lz-1}\int_{\frac{y}2}^{2y}\,dx\,dy\nonumber\\
&&\quad\sim \int_0^\fz|f(y)|^p\,\ytz.
\end{eqnarray}

To estimate ${\rm I}_3$, let $H^{\rm loc}$ be the local Hilbert transform introduced by Andersen and Muckenhoupt \cite{am82}:
$$H^{\rm loc}f(x):=p. v. -\frac{1}{\pi}\int_{\frac x2}^{2x}\frac{f(y)}{x-y}\,dy,\qquad x>0.$$ We write
$$H^{\rm loc}_{\delta_j,t_{j+1}}f(x):=-\frac{1}{\pi}\int_{\frac x2}^{2x}\chi_{B_{\delta_j,\,t_{j+1}}}(y)\frac{f(y)}{x-y}\,dy,\qquad x>0.$$
Then from the mean value theorem, we deduce that
\begin{eqnarray*}
\lf|{\rm I}_3(\delta_j,j)(x)-H^{\rm loc}_{\dz_j,t_{j+1}}f(x)\r|&=&\frac1\pi\lf|\int_{\frac x2}^{2x}\frac{f(y)}{x-y}\frac{y^\lz-x^\lz}{x^{\lz}}\chi_{B_{\delta_j,\,t_{j+1}}}(y)\,dy\r|\\
&\ls& \frac1x\int_{\frac x2}^{2x}|f(y)|\chi_{B_{\delta_j,\,t_{j+1}}}(y)\,dy.\\
\end{eqnarray*}
This implies that
$$\lf\|{\rm I}_3(\delta_j,j)(x)\r\|_{\mathcal E}\ls  \lf\|H^{\rm loc}_{\dz_j,t_{j+1}}f\r\|_{\mathcal E}+\cm_\lz f(x).$$
Observe that
\begin{eqnarray*}
H^{loc}f(x)=H\tilde{f}(x)+\frac{1}{\pi}\int_{0}^{x/2}\frac{f(y)}{x-y}dy+\frac{1}{\pi}\int_{2x}^{\infty}\frac{f(y)}{x-y}dy,
\end{eqnarray*}
where $H$ is the Hilbert transform and
\begin{eqnarray*}
\tilde{f}(x)=\lf\{
\begin{array}{ll}
f(x), \qquad &x>0;\\
0, \qquad &\mbox{otherwise.}
\end{array}\r.
\end{eqnarray*}
 Therefore, we have
\begin{eqnarray*}
H^{loc}_{\delta_j,t_{j+1}}f(x)&&=H_{\delta_j,t_{j+1}}\tilde{f}(x)
+\frac{1}{\pi}\int_{0}^{x/2}\chi_{B_{\delta_j,\,t_{j+1}}}(y)\frac{f(y)}{x-y}dy\\
&&\qquad+\frac{1}{\pi}\int_{2x}^{\infty}\chi_{B_{\delta_j,\,t_{j+1}}}(y)\dfrac{f(y)}{x-y}dy.
\end{eqnarray*}
Then by \eqref{oe} we see that
\begin{eqnarray*}
\mathcal{O}'(H_{\ast}^{loc})f(x)=\|H^{loc}_{\delta_j,t_{j+1}}f\|_{\mathcal E}&&\le\|H_{\delta_j,t_{j+1}}\tilde{f}\|_{\mathcal E}+\frac1\pi\int_{0}^{x/2}\frac{|f(y)|}{|x-y|}dy+\frac1\pi\int_{2x}^{\infty}\frac{|f(y)|}{|x-y|}dy.\\
&&\le \mathcal{O}'(H_{\ast})\tilde{f}(x)+\frac{2}{x\pi}\int_{0}^{x}|f(y)|dy+\frac 2\pi\int_{2x}^{\infty}\frac{|f(y)|}{y}dy.\\
&&\ls \mathcal{O}'(H_{\ast})\tilde{f}(x)+\mathcal{M}f(x)+T_{1}f(x),
\end{eqnarray*}
where $\mathcal{M}f(x)$ is the Hardy-Littlewood maximal function defined as, for any $f\in L^1_{\rm loc}(\rr_+)$ and $x\in \rr_+$,
$$ \mathcal{M}f(x):=\sup_{\gfz{I\subset \rrp}{I\ni x}}\frac1{|I|}\int_I|f(y)|\,dy.$$
Consequently, we have
\begin{eqnarray*}
\oscizp f(x)\lesssim \mathcal{O}'(H_{\ast})\tilde{f}(x)+\mathcal{M}f(x)+T_1f(x)+T_2f(x)+\cm_{\lambda}f(x).
\end{eqnarray*}
From \cite[Theorem 1.5]{gt}, we have $\mathcal{O}'(H_\ast): L^p(\omega)\rightarrow L^p(\omega),\, \omega\in A_p(\rr)$. And from \cite{Mu}, we also have $\mathcal{M}$ is bounded on $L^p(\omega)$. Moreover, from \cite[p.\,286]{LG}, we know that $x^{2\lz}\in A_p$ if and only if $0<2\lz<p-1$. Thus combing \eqref{T1} and \eqref{T2}, we get $\oscizp f(x)$ is bounded on $L^p(\rrp,\dmz)$ for $p\in(1+2\lz, \infty)$.

{\bf Step 2.} By applying the Marcinkiewicz interpolation theorem and Step 1, to finish the proof of Theorem \ref{t-bdd riesz},
it suffices to prove that for any $f\in L^1(\rrp,\,\dmz)$ and $\eta>0$,
\begin{equation}{\label{rf}}
m_{\lz}(\{x\in\rrp: \oscizp f(x)>\eta\})\ls \frac{1}{\eta}\|f\|_{L^1(\rrp,\,\dmz)}.
\end{equation}
The main tool in our proof is the Calder\'on-Zygmund decomposition in Lemma \ref{l-cz}. Let $g,\,b,\,\{b_j\}_j$ and $\{I_j\}_j$ be as Lemma \ref{l-cz}.
Since the operator $\oscizp$ is sublinear, to show \eqref{rf}, it suffices to prove
\begin{equation}{\label{rg}}
m_{\lz}\Big(\Big\{x\in\rrp: |\oscizp(g)(x)|>\frac{\eta}{2}\Big\}\Big)\ls\frac{1}{\eta}\|f\|_{L^1(\rrp,\,\dmz)},
\end{equation}
 and
\begin{equation}\label{rb}
m_{\lz}\Big(\Big\{x\in\rrp: |\oscizp(b)(x)|>\frac{\eta}{2}\Big\}\Big)\ls \frac{1}{\eta}\|f\|_{L^1(\rrp,\,\dmz)}.
\end{equation}

For \eqref{rg}, by the $L^{p}$-boundedness of $\oscizp$ with $p>1+2\lz$ from Step 1 and Lemma \ref{l-cz} {\rm (iii)}, we have
\begin{eqnarray*}
m_{\lz}\Big(\Big\{x\in\rrp: |\oscizp(g)(x)|>\dfrac{\eta}{2}\Big\}\Big)&&\ls \frac{1}{{\eta}^p}\int_{\rrp}\lf[\oscizp(g)(x)\r]^{p}\dmz(x)\\
&&\ls \frac{1}{{\eta}^p}\int_{\rrp}|g(x)|^{p}\dmz(x)\\
&&\ls \frac{1}{{\eta}} \int_{\rrp}|f(x)|\dmz(x).
\end{eqnarray*}
This shows \eqref{rg}.

In what follows, we prove \eqref{rb}.  Let $\tilde{I_{j}}:= 3I_j$ and $\mathcal{\widetilde{I}}:=\bigcup_{j}\tilde{I_{j}}$.
Using the doubling property of $m_{\lz}$ \eqref{volume property-1} and Lemma \ref{l-cz} {\rm (v)}, we write
\begin{eqnarray*}
&&m_{\lz}\Big(\Big\{x\in\rrp:\,|\oscizp(b)(x)|>\dfrac{\eta}{2}\Big\}\Big)\\
&&\qquad\lesssim m_{\lz}(\widetilde{\mathcal{I}})+m_{\lz}\Big(\Big\{x\in\rr_{+}\backslash\widetilde{\mathcal{I}}:\,|\oscizp(b)(x)|>\dfrac{\eta}{2}\Big\}\Big)\\
&&\qquad\lesssim\frac{1}{\eta}\|f\|_{\loz}+m_{\lz}\Big(\Big\{x\in\rr_{+}\backslash\mathcal{\widetilde{I}}:|\oscizp(b)(x)|>\frac{\eta}{2}\Big\}\Big).
\end{eqnarray*}
It remains to estimate the second term on the right of the last inequality. To this end, we first introduce some notations. Let $\dz_i\in (t_{i+1},\,t_{i}]$ and $A_{\dz_i}$ be the interval $(t_{i+1},\,\dz_i]$.
Set $$R_{\Delta_{\lz}}^{A_{\dz_i}}b(x):=\dint_{|x-y|\in A_{\dz_i}}R_{\Delta_{\lz}}(x,y)b(y)\dmz(y).$$
Then, the operator $\oscizp (b)(x)$ can be expressed more conveniently as:
$$\oscizp b(x)=\lf[\sum_{i=1}^{\infty}\sup_{{\dz_i}\in (t_{i+1},\,t_i]}\Big|R_{\Delta_{\lz}}^{A_{\dz_i}}(b)(x)\Big|^{2}\r]^{1/2}.$$
For every $x\in\rrp\backslash\mathcal{\widetilde{I}}$, choose a $\tilde{\dz_i}\in (t_{i+1},\,t_i]$ such that
$$ \oscizp (b)(x)\leq\Big[\sum_{i} 2\Big|R_{\Delta_{\lz}}^{A_{\tilde{\dz_i}}}(b)(x)\Big|^2\Big]^{1/2}\le 2\Big[\sum_{i} \Big|\sum_jR_{\Delta_{\lz}}^{A_{\tilde{\dz_i}}}(b_j)(x)\Big|^2\Big]^{1/2}.$$
Then we only need to prove the following inequality:
\begin{equation}\label{db}
m_{\lz}\Big(\Big\{x\in\rr_{+}\backslash\mathcal{\widetilde{I}}:\Big(\sum_{i} |\sum_jR_{\Delta_{\lz}}^{A_{\tilde{\dz_i}}}(b_j)(x)|^2\Big)^{1/2}>\frac{\eta}{4}\Big\}\Big)\ls \frac{1}{{\eta}} \|f\|_{\loz}.
\end{equation}
For each $x\in \rr_+\setminus\widetilde{\mathcal{I}}$ and $i$,
 $$A_{\tilde{\dz_i}}+x:=(x+t_{i+1},\,x+\tilde{\dz_i}]\cup\lf([x-\tilde{\dz_i},\,x-t_{i+1})\cap \rrp\r).$$   Note that $R_{\Delta_{\lz}}^{A_{\tilde{\dz_i}}}(b_j)(x)$ is nonzero only if, for some $j$, $I_{j}$ lies entirely in $A_{\tilde{\dz_i}}+x$ or $I_{j}$ intersects  the boundary of $A_{\tilde{\dz_i}}+x$. So, we split the family of indexes $j$ in two categories:
$$L_i^1:=\{j:I_{j}\subset x+A_{\tilde{\dz_i}}\} \ \ \ {\rm and}\ \ \   L_i^2:=\{j:I_{j}\nsubseteq x+A_{\tilde{\dz_i}}\, \text{and}\,  I_{j}\cap(x+ A_{\tilde{\dz_i}})\neq \emptyset\}.$$
Thus,
\begin{eqnarray*}
&&\Big[\sum_{i}\Big|\sum_{j}R_{\Delta_{\lz}}^{A_{\tilde{\dz_i}}}(b_j)(x)\Big|^{2}\Big]^{1/2}\\
&&\qquad\le  \Big[\sum_{i}\Big|\sum_{j\in L_i^1}R_{\Delta_{\lz}}^{A_{\tilde{\dz_i}}}(b_j)(x)\Big|^{2}\Big]^{1/2}+\Big[\sum_{i}\Big|\sum_{j\in L_i^2}R_{\Delta_{\lz}}^{A_{\tilde{\dz_i}}}(b_j)(x)\Big|^{2}\Big]^{1/2}.
\end{eqnarray*}
Hence, to prove \eqref{db}, it suffices to show the following two inequalities:
\begin{equation}\label{dbi}
m_{\lz}\Big(\Big\{x\in\rrp\backslash\mathcal{\widetilde{I}}:\Big[\sum_{i}\Big|\sum_{j\in L_i^1}R_{\Delta_{\lz}}^{A_{\tilde{\dz_i}}}(b_j)(x)\Big|^{2}\Big]^{1/2}>\frac{\eta}{8}\Big\}\Big)\lesssim \frac{1}{{\eta}}\|f\|_{L^1(\rrp,\dmz)},
\end{equation}
 and
\begin{equation*}\label{dbo}
m_{\lz}\Big(\Big\{x\in\rrp\backslash\mathcal{\widetilde{I}}:\Big[\dsum_{i}\Big|\sum_{j\in L_i^2}R_{\Delta_{\lz}}^{A_{\tilde{\dz_i}}}(b_j)(x)\Big|^{2}\Big]^{1/2}>\frac{\eta}{8}\Big\}\Big)\lesssim \frac{1}{{\eta}}\|f\|_{L^1(\rrp,\dmz)}.
\end{equation*}

 We first prove \eqref{dbi}. Fix $i\in\nn$ and let $j\in L_{i}^{1}$, that is, $I_{j}\subset x+A_{\tilde{\dz_i}}$. Since by Lemma \ref{l-cz} (ii), $\int_{\rrp}b_j(x)\dmz(x)=0$, we have
\begin{equation}\label{r-b}
R_{\Delta_{\lz}}^{A_{\tilde{\dz_i}}}(b_j)(x)=\int_{\rrp}R_{\Delta_{\lz}}(x,y)b_{j}(y)\dmz(y)=\int_{\rrp}\lf[R_{\Delta_{\lz}}(x,y)-R_{\Delta_{\lz}}(x,y_j)\r]b_{j}(y)\dmz(y),
\end{equation}
where $y_j$ is the center of $I_{j}$. Because for fixed $x$, $x+A_{\tilde{\dz_i}}$ and $x+A_{\tilde{\dz}_{k}}$ are disjoint for $i\neq k$, we see that for $j\in L_{i}^{1}$, $I_j$ and $x+A_{\tilde{\dz}_{k}}$ are disjoint. Therefore, by \eqref{cz kernel condition-2} and \eqref{r-b},
\begin{eqnarray*}
\Big[\sum_{i}\Big|\sum_{j\in L_i^1}R_{\Delta_{\lz}}^{A_{\tilde{\dz_i}}}(b_j)(x)\Big|^{2}\Big]^{1/2}&&\le \sum_{i}\Big|\dsum_{j\in L_i^1}R_{\Delta_{\lz}}^{A_{\tilde{\dz_i}}}(b_j)(x)\Big|\\
&&\lesssim \sum_{j}\dint_{\rrp} |R_{\Delta_{\lz}}(x,y)-R_{\Delta_{\lz}}(x,y_j)||b_j(y)|\dmz(y)\\
&&\lesssim\sum_{j}\dint_{\rrp}\dfrac{|y-y_j|}{|x-y|}\dfrac{1}{\mxy}|b_{j}(y)|\dmz(y)\\
\end{eqnarray*}
 Thus we have
\begin{eqnarray*}
&&m_{\lz}\Big(\Big\{x\in\rrp\backslash\mathcal{\widetilde{I}}:\Big[\sum_{i}\Big|\sum_{j\in L_i^1}R_{\Delta_{\lz}}^{A_{\tilde{\dz_i}}}(b_j)(x)\Big|^{2}\Big]^{1/2}>\frac{\eta}{8}\Big\}\Big)\\
&&\qquad\lesssim \frac{1}{\eta}\int_{\rrp\backslash\mathcal{\widetilde{I}}}\sum_j\int_{\rrp}\dfrac{|y-y_j|}{|x-y|}\dfrac{1}{\mxy}|b_{j}(y)|\dmz(y)\dmz(x)\\
&&\qquad\lesssim \frac{1}{\eta}\sum_{j}\int_{I_j}|b_j(y)|\int_{\rrp\backslash\tilde{I_j}}\dfrac{|y-y_j|}{|x-y|}\dfrac{1}{\mxy}\dmz(x)\dmz(y)\\
&&\qquad\lesssim \frac{1}{\eta}\sum_{j}\int_{I_j}|b_j(y)|\sum_{k=1}^{\infty}\frac{|I_j|}{3^k|I_j|}\int_{3^{k+1}I_j\setminus3^kI_j}\frac{\dmz(x)}{\mxy}\dmz(y).
\end{eqnarray*}
Note that for any $x\in \rrp\setminus \mathcal{\widetilde{I}}$, \, $y,\,y_j\in I_{j}$, we have $|x-y|\sim |x-y_j|$ and

\begin{equation}\label{mxyj}
\mxy\sim  m_{\lz}(I(x,|x-y_j|))\sim m_{\lz}(I(y_j,|x-y_j|)).
\end{equation}
Thus, we get
\begin{eqnarray*}
\dint_{3^{k+1}I_j\backslash 3^kI_j}\dfrac{\dmz(x)}{\mxy}\lesssim\dfrac{m_{\lz}(3^{k+1}I_j)}{m_{\lz}(3^{k}I_j)}\ls 1.
\end{eqnarray*}
Consequently, by this fact and Lemma \ref{l-cz} (iv), we conclude that
$$m_{\lz}\Big(\Big\{x\in\rr_{+}\backslash\mathcal{\tilde{I}}: \,\Big(\sum_i\Big|\sum_{j\in L_i^1}R_{\Delta_{\lz}}^{A_{\tilde{\dz_i}}}(b_j)(x)\Big)>\dfrac {\eta}{8}\Big\}\Big)\lesssim \frac{1}{{\eta}}\|f\|_{\loz}.$$
This completes the proof of \eqref{dbi}.

 For the part $L^2_i$, a simple geometrical inspection via Lemma \ref{l-cz}(vi) shows that $L_i^2$ contains at most finite $j$'s for any $i$. It then follows that
 \begin{eqnarray*}
 \sum_i\lf|\dsum_{j\in L_i^2}R_{\Delta_{\lz}}^{A_{\tilde{\dz_i}}}(b_j)(x)\r|^2&&\lesssim \sum_i\sum_{j\in L_i^2}\lf|R_{\Delta_{\lz}}^{A_{\tilde{\dz_i}}}(b_j)(x)\r|^2\\
 &&\ls \sum_j\sum_i\lf|R_{\Delta_{\lz}}^{A_{\tilde{\dz_i}}}(b_j)(x)\r|^2\\
 &&\ls
 \sum_j\lf(\sum_i\lf|R_{\Delta_{\lz}}^{A_{\tilde{\dz_i}}}(b_j)(x)\r|\r)^2.
 \end{eqnarray*}
Thus
\begin{eqnarray*}
&&m_{\lz}\Big(\Big\{x\in\rrpi: \Big(\dsum_{i}|\dsum_{j\in L_{i}^{2}} R_{\Delta_{\lz}}^{A_{\tilde{\dz_i}}}(b_j)(x)|^{2}\Big)^{1/2}>\frac {\eta}{8}\Big\}\Big)\\
&&\qquad\le
m_{\lz}\Big(\Big\{x\in\rrpi: \dsum_{i}|\dsum_{j\in L_{i}^{2}} R_{\Delta_{\lz}}^{A_{\tilde{\dz_i}}}(b_j)(x)|^{2}>\frac {{\eta}^2}{64}\Big\}\Big)\\
&&\qquad\lesssim \frac{1}{{\eta}^2}\dsum_j\dint_{\rrp\setminus \tilde{I}_j}\lf(\dsum_{i}\lf| R_{\Delta_{\lz}}^{A_{\tilde{\dz_i}}}(b_j)(x)\r|\r)^{2}\dmz(x).
\end{eqnarray*}
Let $x\in \rr_{+}\setminus\tilde{I}_j$ for some $j$. Then by \eqref{cz kernel condition-1} and \eqref{mxyj}, we have
\begin{eqnarray*}
\sum_{i}| R_{\Delta_{\lz}}^{A_{\tilde{\dz_i}}}(b_j)(x)|^{2}&&\le \sum_i\lf[\int_{|x-y|\in A_{\tilde{\dz_i}}}|R_{\Delta_{\lz}}(x,y)||b_j(y)|\dmz(y)\r]^2\\
 &&\lesssim \sum_i\lf[\int_{|x-y|\in A_{\tilde{\dz_i}}}\dfrac{|b_j(y)|}{\mxy}\dmz(y)\r]^2\\
 &&\lesssim\dfrac{1}{[ m_{\lz}(I(y_j,|x-y_j|))]^2}\int_{\rrp}|b_j(y)|\dmz(y)\\
 &&\quad\times \sum_i\int_{|x-y|\in A_{\tilde{\dz_i}}}\dfrac{|b_j(y)|}{\mxy}\dmz(y)\\
 &&\lesssim\dfrac{1}{[ m_{\lz}(I(y_j,|x-y_j|))]^2}\lf[\int_{\rrp}|b_j(y)|\dmz(y)\r]^2.
\end{eqnarray*}
Therefore, by (ii) and (iv) of Lemma \ref{l-cz} , we get
\begin{eqnarray*}
&&\int_{\rrp\setminus \tilde{I}_j}\Big(\sum_i\lf| R_{\Delta_{\lz}}^{A_{\tilde{\dz_i}}}(b_j)(x)\r|^2\Big)\dmz(x)\\
&&\quad\lesssim\int_{\rrp\setminus \tilde{I}_j}\dfrac{1}{[ m_{\lz}(I(y_j,|x-y_j|))]^2}\Big(\int_{\rrp}|b_j(y)|\dmz(y)\Big)^2\dmz(x)\\
&&\quad\lesssim\sum_{k=1}^{\infty}\int_{3^{k+1}I_j\backslash3^kI_j}\dfrac{\eta m_{\lz}(I_j)}{[ m_{\lz}(I(y_j,|x-y_j|))]^2}\dmz(x)
\int_{\rrp}|b_j(y)|\dmz(y)\\
&&\quad\lesssim \eta\int_{\rrp}|b_j(y)|\dmz(y)\sum_{k=1}^{\infty}\dfrac{m_{\lz}(I_j)m_{\lz}(3^{k+1}I_j)}{[m_{\lz}(3^kI_j)]^2}\\
&&\quad\lesssim\eta\int_{\rrp}|b_j(y)|\dmz(y).
\end{eqnarray*}
Consequently, we get $$m_{\lz}\Big(\Big\{x\in\rrpi: \Big[\sum_{j}\Big|\sum_{i\in L_{i}^{2}} R_{\Delta_{\lz}}^{A_{\tilde{\dz_i}}}(b_j)(x)\Big|^{2}\Big]^{1/2}>\frac {\eta}{4}\Big\}\Big)\lesssim \frac{1}{\eta}\|f\|_{\loz},$$
which finishes the proof of weak type (1,\,1) estimation of $\oscizp$. Theorem 1.1 is proved.
\end{proof}

\section{Proof of  Theorem \ref{t-bmo riesz}}\label{s3}

As the proof of Theorem \ref{t-bdd riesz}, by similarity, we only consider $\osciz$. The proof for $\varoz$ is similar and omitted.

 Fix $f\in \liz$ and define $f_1(y):=f(y)\chi_{4I}$ and $f_2(y):=f(y)\chi_{\rrp\backslash 4I}$, where $I:=I(x_0,\,r)$.
 From  the H\"older inequality, Theorem \ref{t-bdd riesz} and \eqref{volume property-1}, we deduce that
\begin{align}\label{o-f1}
\dint_I|\osciz(f_1)(x)|\dmz(x)&\le[m_\lz(I)]^\frac12\lf[\dint_{\rrp}|\osciz (f_1)(x)|^2\dmz(x)\r]^\frac12\nonumber\\
&\ls [m_\lz(I)]^\frac12\lf[\dint_{4I}|f(x)|^2\dmz(x)\r]^\frac12\nonumber\\
&\lesssim m_{\lz}(I)\|f\|_{\liz}.
\end{align}
Then $\osciz(f_1)(x)<\infty,\, a.\,e.\, x\in I$.  According to the assumption,  we may choose $x_1\in I$ such that $\osciz(f_2)(x_1)<\infty$.

By the fact that $\osciz$ is sublinear, \eqref{o-f1}, 
we write
\begin{eqnarray*}
&&\frac{1}{m_{\lz}(I)}\int_{I}\lf|\osciz(f)(x)-\osciz(f_2)(x_1)\r|\dmz(x)\nonumber\\
&&\quad\le\frac{1}{m_{\lz}(I)}\int_{I}\osciz(f_1)(x)\,\dmz(x)\nonumber\\
&&\qquad+\frac{1}{m_{\lz}(I)}\int_{I}\lf|\osciz(f_2)(x)-\osciz(f_2)(x_1)\r|\dmz(x)\nonumber\\
&&\quad\ls\|f\|_{\liz}+\frac{1}{m_{\lz}(I)}\int_{I}\lf|\osciz(f_2)(x)-\osciz(f_2)(x_1)\r|\,\dmz(x).
\end{eqnarray*}
Thus,  to finish the proof of Theorem \ref{t-bmo riesz}, it suffices to show that
\begin{eqnarray*}\label{u2}
\frac{1}{m_{\lz}(I)}\int_{I}\lf|\osciz(f_2)(x)-\osciz(f_2)(x_1)\r|\,\dmz(x)\ls \|f\|_\linz.
\end{eqnarray*}

Assume that $\osciz$ is defined by a given sequence $\{t_i\}_i$ decreasing and  converging to zero.
Now, for $x\in I$, and $I_i:=(t_{i+1}, t_i]$, $i\in\nn$, by \eqref{oci-o'ci} and \eqref{oe}, we write
\begin{eqnarray*}
&&\lf|\osciz(f_2)(x)-\osciz(f_2)(x_1)\r|\\
&&\quad\le \lf(\sum_{i=1}^\fz\sup_{t_{i+1}\le \ez_{i+1}< \ez_i \le t_i}\lf|\lf[R_{\Delta_\lz,\,\ez_i}f_2(x)-R_{\Delta_\lz,\,\ez_{i+1}}f_2(x)\r]\r.\r.\\
&&\quad\quad\quad\lf.\lf.-\lf[R_{\Delta_\lz,\,\ez_i}f_2(x_1)-R_{\Delta_\lz,\,\ez_{i+1}}f_2(x_1)\r]\r|^2\r)^{1/2}\\
&&\quad\ls \lf(\sum_{i=1}^\fz\sup_{t_{i+1}< \dz_i \le t_i}\lf|\lf[R_{\Delta_\lz,\,t_{i+1}}f_2(x)-R_{\Delta_\lz,\,\dz_i}f_2(x)\r]\r.\r.\\
&&\quad\quad\quad\lf.\lf.-\lf[R_{\Delta_\lz,\,t_{i+1}}f_2(x_1)-R_{\Delta_\lz,\,\dz_{i}}f_2(x_1)\r]\r|^2\r)^{1/2}\\
&&\quad=\Big\|\Big\{\int_{\{t_{i+1}<|x-y|\le\dz_i\}}\riz(x,y)f_2(y)\,\dmz(y)\\
&&\quad\quad\,\,-\int_{\{t_{i+1}<|x_1-y|\le\dz_i\}}\riz(x_1,y)f_2(y)\,\dmz(y)\Big\}_{\dz_i\in I_i,\,i\in \mathbb{N}}\Big\|_{\mathcal E}\\
&&\quad\le\lf\|\lf\{\int_{\{t_{i+1}<|x-y|\le\dz_i\}}[\riz(x,y)-\riz(x_1,y)]f_2(y)\,\dmz(y)\r\}_{\dz_i\in I_i,\,i\in \mathbb{N}}\r\|_{\mathcal E}\\
&&\qquad\,\, +\lf\|\lf\{\int_{\rrp}\lf[\chi_{\{t_{i+1}<|x-y|\le\dz_i\}}(y)-
\chi_{\{t_{i+1}<|x_1-y|\le\dz_i\}}(y)\r]\riz(x_1,y)f_2(y)\,\dmz(y)\r\}_{\dz_i\in I_i,\,i\in \mathbb{N}}\r\|_{\mathcal E}\\
&&\quad=:\rm {D}_1+\rm {D}_2.
\end{eqnarray*}
Notice that
$$\|\{\chi_{\{t_{i+1}<|x-y|\le\dz_i\}}(y)\}_{\dz_i\in I_i,\,i\in \mathbb{N}}\|_{\mathcal E}\le 1.$$
 By this fact, Mikowski's inequality, \eqref{cz kernel condition-2} and \eqref{volume property-1}, we have
\begin{eqnarray*}
\rm {D}_1&&\le \int_{\rrp}\|\{\chi_{\{t_{i+1}<|x-y|\le\dz_i\}}(y)\}_{\dz_i\in J_i,i\in \mathbb{N}}\|_{\mathcal E}|\riz(x,y)-\riz(x_1,y)||f_2(y)|\dmz(y)\\
&&\le\int_{\rrp}|\riz(x,y)-\riz(x_1,y)||f_2(y)|\dmz(y)\\
&&\lesssim\sum_{k=2}^{\infty}\int_{2^{k+1}I\backslash2^kI}\frac{|x-x_1||f(y)|}{|x-y|m_{\lz}(I(y,|x-y|))}\dmz(y)\\
&&\lesssim\|f\|_{\liz}\sum_{k=2}^{\infty}\dfrac{|I|}{2^k|I|}\int_{2^{k+1}I\backslash2^kI}\dfrac{\dmz(y)}{m_{\lz}(I(y,|x-y|))}\\
&&\lesssim\|f\|_{\liz}\sum_{k=2}^{\infty}\dfrac{1}{2^k}\dfrac{m_{\lz}(2^{k+1}I)}{m_{\lz}(2^kI)}\\
&&\lesssim\|f\|_{\liz},
\end{eqnarray*}
where in the third-to-last inequality, we use the fact that for $x,x_0\in I$, $y\in \rrp\setminus 4I$, $|x-y|\sim|x_0-y|$,
and in the second-to-last inquality,
\begin{equation}\label{Ix0x}
m_{\lz}(I(y,|x-y|))\sim m_{\lz}(I(y,|x_0-y|))\sim m_{\lz}(I(x_0,|x_0-y|)).
\end{equation}

For ${\rm {D}_2}$, note that the integral
\begin{equation*}
{\rm E}:=\dint_{\rrp}|\chi_{\{t_{i+1}<|x-y|\le\dz_i\}}(y)-\chi_{\{t_{i+1}<|x_1-y|\le\dz_i\}}(y)||\riz(x_1,y)||f_2(y)|\dmz(y)\neq 0
\end{equation*}
 only if either
 $$\chi_{\{t_{i+1}<|x-y|\le\dz_i\}}(y)=1\quad {\rm and}\quad \chi_{\{t_{i+1}<|x_1-y|\le\dz_i\}}(y)=0,$$
  or viceversa, that is,
  $$\chi_{\{t_{i+1}<|x-y|\le\dz_i\}}(y)=0 \quad {\rm and}\quad \chi_{\{t_{i+1}<|x_1-y|\le\dz_i\}}(y)=1.$$
Equivalently, if ${\rm E}\not=0$, at least one of the following four statements holds:
\begin{itemize}
 \item [{\rm  (i)}]$t_{i+1}<|x-y|\le\dz_i$ and $|x_1-y|\le t_{i+1}$;
  \item [{\rm (ii)}]$t_{i+1}<|x-y|\le\dz_i$ and $|x_1-y|> \dz_i$;
  \item[{\rm (iii)}] $|x-y|\le t_{i+1}$ and $t_{i+1}<|x_1-y|\le\dz_i$ ;
  \item [{\rm (iv)}]$|x-y|> \dz_i$ and $t_{i+1}<|x_1-y|\le\dz_i$.
\end{itemize}

Since  $|x-x_1|<2r$, we observe that in case (i),
$$t_{i+1}<|x-y|\le |x-x_1|+|x_1-y|<t_{i+1}+2r ;$$
in case (ii),
$$\dz_i<|x_1-y|\le |x_1-x|+|x-y|<\dz_i+ 2r;$$
in case (iii),
$$t_{i+1}<|x_1-y|<t_{i+1}+2r;$$
and in case (iv),
$$\dz_i<|x-y|<\dz_i +2r.$$
 Then, we write
\begin{eqnarray*}
{\rm E}&&\lesssim\int_{\rrp}\chi_{\{t_{i+1}<|x-y|\le\dz_i\}}(y)\chi_{\{t_{i+1}<|x-y|<t_{i+1}+2r\}}(y)|\riz(x_1,y)||f_2(y)|\,\dmz(y)\\
&&\quad+\int_{\rrp}\chi_{\{t_{i+1}<|x-y|\le\dz_i\}}(y)\chi_{\{\dz_i<|x_1-y|<\dz_i+2r\}}(y)|\riz(x_1,y)||f_2(y)|\,\dmz(y)\\
&&\quad+\int_{\rrp}\chi_{\{t_{i+1}<|x_1-y|\le\dz_i\}}(y)\chi_{\{t_{i+1}<|x_1-y|<t_{i+1}+2r\}}(y)|\riz(x_1,y)||f_2(y)|\,\dmz(y)\\
&&\quad+\int_{\rrp}\chi_{\{t_{i+1}<|x_1-y|\le\dz_i\}}(y)\chi_{\{\dz_i<|x-y|<\dz_i+2r\}}(y)|\riz(x_1,y)||f_2(y)|\,\dmz(y)\\
&&=:{\rm {J}}_1+{\rm{J}}_2+{\rm{J}}_3+{\rm{J}}_4.
\end{eqnarray*}

We now estimate ${\rm J_1}$ and claim that
 \begin{eqnarray}\label{esti for J1}
{{\rm{J}}_1}\ls \lf[\int_{\rrp}\chi_{\{t_{i+1}<|x-y|\le\dz_i\}}(y)\frac{r}{|x_0-y|}
\frac{|f_2(y)|^2}{m_{\lz}(I(x_0,|x_0-y|))}\,\dmz(y)\r]^{1/2}.
\end{eqnarray}
Indeed, observe that  for any $y\in \rrp\setminus 4I$ and $x\in I$,
 \begin{equation}\label{metric equiv}
 |x-y|>3r, \,\,|x_1-y|>3r,\,\,{\rm and}\,\,|x_1-y|/3<|x-y|<5|x_1-y|/3.
 \end{equation}
  Moreover, if $x,\,x_1\in I$ and $r\ge t_{i+1}$, $i\in\mathbb{N}$, we have
$$\{y\in\rrp\setminus 4I:\,\,t_{i+1}<|x-y|<t_{i+1}+2r\}\subset \{y\in\rrp\setminus 4I:\,\,|x-y|<3r\}=\emptyset.$$
This means ${\rm{J}}_{1}=0$ and \eqref{esti for J1} holds.

In the following, we assume  that $r<t_{i+1}$. From this assumption, we further deduce that
for $x\in I$,
\begin{eqnarray*}\label{meas upp t}
0<m_{\lz}(I(x,t_{i+1}+2r))-m_{\lz}(I(x,t_{i+1}))\ls (x+t_{i+1}+2r)^{2\lz} r\ls (x+t_{i+1})^{2\lz} r.
\end{eqnarray*}

 By this fact,  together with \eqref{cz kernel condition-1}, \eqref{volume property-1}, H\"{o}lder's inequality and the fact that for
  any $x\in I,\, y\in \rrp\setminus 4I$,
\begin{eqnarray}\label{mx-sim}
m_{\lz}(I(x_1,\,|x_1-y|))\sim m_{\lz}(I(y,\,|x_1-y|))\sim m_{\lz}(I(y,\,|x-y|))\sim m_{\lz}(I(x,\,|x-y|)),
\end{eqnarray}
we see that
 \begin{eqnarray*}
{{\rm{J}}_1}&&\le\lf[m_{\lz}(I(x,t_{i+1}+2r))-m_{\lz}(I(x,t_{i+1}))\r]^{1/2}\\
&&\quad\times\lf(\int_{\rrp}\chi_{\{t_{i+1}<|x-y|\le\dz_i\}}(y)|\riz(x_1,y)|^2|f_2(y)|^2\dmz(y)\r)^{1/2}\\
%
&&\ls\lf[ (x+t_{i+1})^{2\lz} r \int_{\rrp}\chi_{\{t_{i+1}<|x-y|\le\dz_i\}}(y)\frac{|f_2(y)|^2}{m_{\lz}(I(x_1,|x_1-y|))^2}\,\dmz(y)\r]^{1/2}\\
&&\lesssim \lf[  \int_{\rrp}\chi_{\{t_{i+1}<|x-y|\le\dz_i\}}(y)\frac{|f_2(y)|^2}{m_{\lz}(I(x_1,|x_1-y|))}
\frac{(x+t_{i+1})^{2\lz}r}{m_{\lz}(I(x,|x-y|))}\,\dmz(y)\r]^{1/2}\\
&&\lesssim \lf[\int_{\rrp}\chi_{\{t_{i+1}<|x-y|\le\dz_i\}}(y)\frac{|f_2(y)|^2}{m_{\lz}(I(x,|x-y|))}\frac{(x+|x-y|)^{2\lz}r}{(x^{2\lz}+|x-y|^{2\lz})|x-y|}
\,\dmz(y)\r]^{1/2}\\
&&\ls \lf[\int_{\rrp}\chi_{\{t_{i+1}<|x-y|\le\dz_i\}}(y)\frac{|f_2(y)|^2}{m_{\lz}(I(x_0,|x_0-y|))}\frac{r}{|x_0-y|}
\,\dmz(y)\r]^{1/2}.
\end{eqnarray*}
This implies \eqref{esti for J1}.

Now we estimate ${\rm J}_2$. By \eqref{metric equiv}, we see that
for $x_1\in I$ and $r\ge \dz_i$, $i\in\mathbb{N}$, we have
$$\{y\in\rrp\setminus 4I:\,\,\dz_i<|x-y|<\dz_i+2r\}\subset \{y\in\rrp\setminus 4I:\,\,|x-y|<3r\}=\emptyset.$$
which implies that ${\rm{J}}_{2}=0$.  Therefore, in what follows, we assume $r<\dz_i$. Observe that
\begin{eqnarray*}\label{meas upp delta}
0<m_{\lz}(I(x_1,\dz_i+2r))-m_{\lz}(I(x_1,\dz_i))\ls (x_1+\dz_i+2r)^{2\lz} r\ls (x_1+\dz_i)^{2\lz} r.
\end{eqnarray*}
 By this fact,  H\"{o}lder's inequality, \eqref{cz kernel condition-1}, \eqref{mx-sim}, \eqref{metric equiv} and \eqref{Ix0x}, we have
\begin{eqnarray*}
{{\rm{J}}_2}&&\le\lf[m_{\lz}(I(x_1,\dz_i+2r))-m_{\lz}(I(x_1,\dz_i))\r]^{1/2}\\
&&\quad\times\lf(\int_{\rrp}\chi_{\{\dz_i<|x_1-y|<\dz_i+2r\}}(y)\chi_{\{t_{i+1}<|x-y|\le\dz_i\}}(y)|\riz(x_1,y)|^2|f_2(y)|^2\dmz(y)\r)^{1/2}\\
&&\ls\lf[ \int_{\rrp}\chi_{\{\dz_i<|x_1-y|<\dz_i+2r\}}(y)\chi_{\{t_{i+1}<|x-y|\le\dz_i\}}(y)\frac{|f_2(y)|^2 (x_1+\dz_i)^{2\lz} r}{m_{\lz}(I(x_1,|x_1-y|))^2}\,\dmz(y)\r]^{1/2}\\
&&\lesssim \lf[  \int_{\rrp}\chi_{\{t_{i+1}<|x-y|\le\dz_i\}}(y)\frac{|f_2(y)|^2}{m_{\lz}(I(x_1,|x_1-y|))}
\frac{(x_1+|x_1-y|)^{2\lz}r}{m_{\lz}(I(x_1,|x_1-y|))}\,\dmz(y)\r]^{1/2}\\
&&\lesssim \lf[\int_{\rrp}\chi_{\{t_{i+1}<|x-y|\le\dz_i\}}(y)\frac{|f_2(y)|^2}{m_{\lz}(I(x_1,|x_1-y|))}\frac{(x_1+|x_1-y|)^{2\lz}r}{(x_1^{2\lz}+|x_1-y|^{2\lz})|x_1-y|}
\,\dmz(y)\r]^{1/2}\\
&&\ls \lf[\int_{\rrp}\chi_{\{t_{i+1}<|x-y|\le\dz_i\}}(y)
\frac{|f_2(y)|^2}{m_{\lz}(I(x_0,|x_0-y|))}\frac{r}{|x_0-y|}\,\dmz(y)\r]^{1/2}.
\end{eqnarray*}

Similarly,  we have
\begin{eqnarray*}
{{\rm{J}}_3}&&\le\lf[m_{\lz}(I(x_1,t_{i+1}+2r))-m_{\lz}(I(x_1,t_{i+1}))\r]^{1/2}\\
&&\quad\times\lf(\int_{\rrp}\chi_{\{t_{i+1}<|x_1-y|<t_{i+1}+2r\}}(y)\chi_{\{t_{i+1}<|x_1-y|\le\dz_i\}}(y)|\riz(x_1,y)|^2|f_2(y)|^2\,\dmz(y)\r)^{1/2}\\
&&\lesssim \lf[\int_{\rrp}\chi_{\{t_{i+1}<|x_1-y|\le\dz_i\}}(y)\frac{|f_2(y)|^2}{m_{\lz}(I(x_0,|x_0-y|))}\frac r{|x_0-y|}\,\dmz(y)\r]^{1/2},
\end{eqnarray*}
and
\begin{eqnarray*}
{{\rm{J}}_4}&&\le\lf[m_{\lz}(I(x,\dz_i+2r))-m_{\lz}(I(x,\dz_i))\r]^{1/2}\\
&&\quad\times\lf(\int_{\rrp}\chi_{\{\dz_i<|x-y|<\dz_i+2r\}}(y)\chi_{t_{i+1}<|x_1-y|\le\dz_i\}}(y)|\riz(x_1,y)|^2|f_2(y)|^2\,\dmz(y)\r)^{1/2}\\
&&\lesssim \lf[\int_{\rrp}\chi_{\{t_{i+1}<|x_1-y|\le\dz_i\}}(y)\frac{|f_2(y)|^2}{m_{\lz}(I(x_0,|x_0-y|))}\frac r{|x_0-y|}\,\dmz(y)\r]^{1/2}.
\end{eqnarray*}

Now returning  to the estimate of $\rm{D}_2$ and combing the estimates of ${\rm J}_1, {\rm J}_2,{\rm J}_3$ and ${\rm J}_4$, we can write
\begin{eqnarray*}
{{\rm{D}}_2}&&\lesssim  \Big\|\Big\{\Big[\int_{\rrp}\chi_{\{t_{i+1}<|x-y|\le\dz_i\}}(y)\frac{|f_2(y)|^2}{m_{\lz}(I(x_0,|x_0-y|))}\frac r{|x_0-y|}\dmz(y)\Big]^{1/2}\Big\}_{\dz_i\in I_i,\,i\in\mathbb{N}}\Big\|_{\mathcal E}\\
&&\qquad+ \Big\|\Big\{\Big[\int_{\rrp}\chi_{\{t_{i+1}<|x_1-y|\le\dz_i\}}(y)\frac{|f_2(y)|^2}{m_{\lz}(I(x_0,|x_0-y|))}\frac r{|x_0-y|}\dmz(y)\Big]^{1/2}\Big\}_{\dz_i\in I_i,\,i\in\mathbb{N}}\Big\|_{\mathcal E}\\
&&=:{\rm{D}}_{21}+{\rm{D}}_{22}.
\end{eqnarray*}
A straight calculation involving \eqref{volume property-1} implies that
\begin{eqnarray*}
{\rm D}_{21}&&=\Big[\sum_{i\in \mathbb{N}}\sup_{\dz_i\in I_i}\int_{\rrp}\chi_{\{t_{i+1}<|x-y|\le\dz_i\}}(y)\frac{|f_2(y)|^2}{m_{\lz}(I(x_0,|x_0-y|))}\frac r{|x_0-y|}\,\dmz(y)\Big]^{1/2}\\
&&\lesssim\Big[\int_{\rrp}\frac{|f_2(y)|^2}{m_{\lz}(I(x_0,|x_0-y|))}\frac r{|x_0-y|}\,\dmz(y)\Big]^{1/2}\\
&&\lesssim\Big[\sum_{k=2}^{\infty}\int_{2^{k+1}I\backslash 2^kI}\frac{|f(y)|^2}{m_{\lz}(I(x_0,|x_0-y|))}\frac r{|x_0-y|}\,\dmz(y)\Big]^{1/2}\\
&&\lesssim \|f\|_{\liz}\lf[\sum_{k=2}^{\infty}2^{-k}\dfrac{m_{\lz}(2^{k+1}I)}{m_{\lz}(2^kI)}\r]^\frac12\\
&&\sim \|f\|_{\liz}.
\end{eqnarray*}
Similarly, ${{\rm{D}}_{22}}\lesssim \|f\|_{\liz}.$
Consequently, we have
\begin{equation*}
{{\rm{D}}_{2}}\lesssim \|f\|_{\liz}.
\end{equation*}
Combining the estimate for ${\rm{D}}_1$, we complete the proof of Theorem \ref{t-bmo riesz}.

\bigskip

{\bf Acknowledgments}

This work is done during Dongyong Yang's visit to Macquarie University. He would like to
thank Professors Xuan Thinh Duong and Ji Li for their generous help.

\bigskip

Huoxiong Wu

School of Mathematical Sciences, Xiamen University, Xiamen 361005,  China
\smallskip

{\it E-mail}: \texttt{huoxwu@xmu.edu.cn}

\bigskip

Dongyong Yang


School of Mathematical Sciences, Xiamen University, Xiamen 361005,  China

\smallskip

{\it E-mail}: \texttt{dyyang@xmu.edu.cn }

\bigskip

Jing Zhang

School of Mathematical Sciences, Xiamen University, Xiamen 361005,  China

School of Mathematics and Statistics, Yili Normal College, Yining Xinjiang 835000, China
\smallskip

{\it E-mail}: \texttt{zjmath66@126.com}

\end{document}